\documentclass[12pt]{article}

\newlength{\stefan}
\usepackage{amsmath, amsthm, amsfonts, makeidx}
\usepackage{multind}
\DeclareMathSymbol{\subsetneq}{\mathord}{AMSb}{"26}

\newtheorem{lemma}{Lemma}[section]

\newtheorem{theorem}[lemma]{Theorem}

\newtheorem{proposition}[lemma]{Proposition}
\newtheorem{corollary}[lemma]{Corollary}
\theoremstyle{definition}

\newtheorem{definition}[lemma]{Definition}
\newtheorem{example}[lemma]{Example}

\newtheorem{remark}[lemma]{Remark}

\newcommand{\lp}{\longrightarrow}

\newcommand{\mb}{\mathbb}

\newcommand{\FF}{k}

\newcommand{\C}{\mb{C}}

\newcommand{\Z}{\mb{Z}}
\newcommand{\N}{\mb{N}}

\renewcommand{\a}{\mathfrak{a}}

\newcommand{\desda}{\Longleftrightarrow}

\renewcommand{\deg}{\operatorname{deg}}

\newcommand{\GA}{\textup{GA}}
\newcommand{\MA}{\textup{MA}}
\newcommand{\TA}{\textup{TA}}

\title{Computing preimages of points and curves under polynomial maps}
\author{
Michiel de Bondt~~~~~~~~Stefan Maubach\footnote{Funded by Veni-grant of council for the
physical sciences, Netherlands Organisation for scientific research (NWO)}\\ 
\small M.deBondt@math.ru.nl~~~~~~~~s.maubach@math.ru.nl\\
\small
Radboud University Nijmegen\\
\small Toernooiveld 1, The Netherlands\\ 
}

\begin{document}

\maketitle

\begin{abstract}
In this paper, we give two algorithms to compute preimages of curves under polynomial endomorphisms. In particular, this gives an efficient way of computing preimages of points. Furthermore, we explain the abstract setting under which one can iteratively compute the inverse of a polynomial automorphism. 
\end{abstract}

\section{Notations}

Let $R$ be a  commutative ring with one.

We write $\MA_{n}(R)$ as the set of polynomial maps $R^n\lp R^n$. $\GA_n(R)$ is the subset of $\MA_n(R)$ of invertible polynomial maps.

Define $A:=R^{[n]}:=R[X_1,\ldots,X_n]$. Write $I$ for the identity map on $R^n$.
We will use the notation $\FF$ for any field.

\section{Introduction and motivation}

If $F\in \GA_n(\FF)$ then there are several algorithms to compute the inverse. Essen's algorithm (see \cite{Essen90}) uses Groebner bases
to directly compute the inverse. This algorithm is in general not very efficient unless in low dimensions and also $F$ of low degree. 
In dimension $n=2$ there exist several other algorithms \cite{AdEs, Dick}, which are actual algorithms that decide in finite time if the map is invertible. These algorithms are due to the fact that in dimension $n=2$ the automorphism group is understood by the Jung-van der Kulk-theorem, and are rather efficient.

An ad-hoc way of computing the inverse of a map is computing its formal power series inverse step-by step. For this, bring your map on the form $F=I+H$ where $H$ has no linear or affine part. Any such endomorphism has an inverse $G$ in the formal power series ring $\FF[[X_1,\ldots,X_n]]$ of the form
$G=I-K$ where $K$ has no linear or affine part. What one can do is start computing the coefficients of $G$ from the lowest degree and up: if the coefficients of $G$ are known up to degree $d$, then the coefficients of degree $d+1$ can be computed since 
$F(G)$ is the identity up to degree $d$, and the part of degree $d+1$ fixes the coefficients of $G$ of degree $d+1$. 
In case $F$ is invertible, this procedure at some degree yields the polynomial inverse (the computation may continue but will only yield zero coefficients from that degree on). The efficiency of this approach is sometimes better, sometimes worse as the Essen-algorithm (depending on implementation, size and degree of the automorphism, etc.).

Another way of computing the inverse of a map is decomposing the map in simpler invertible maps. So far the only case where this works is in dimension two over a field, though the non-field case saw some progress through the recent work of Umirbaev-Shestakov \cite{US1, US2}. (This work provides  an algorithm to check if $F\in \TA_2(\C[Z])$, and gives a decomposition in this case.)

The obstructions in the above algorithms are obvious. There is one that we would like to mention, which is that the inverse of $F\in \GA_n(\FF)$ may 
have much larger degree than $F$ and may contain a huge number of nonzero coefficients. If $\deg(F)=d$, then $\deg(F^{-1})\leq d^{n-1}$ and this bound can be attained easily (possibly the bound is attained by a generic automorphism, even). This means that an inverse might not fit in any computer in explicit form, and one would actually require a decomposition into simpler automorphisms.\\

{\bf Our results:}\\

In this paper we address some of the above issues, mainly by focusing on computing preimages of points, in stead of computing the inverse directly. 
First, a variation on the above power-series computation of the inverse is done in section \ref{three}. 
In section \ref{four} we show how this viewpoint can be used 
to efficiently compute preimages of polynomial automorphism or even {\em endomorphisms}, without actually computing the inverse. 
We give two algorithms to do this: First, the already known algorithm of van den Essen, which uses Groebner bases, can be used.
Its efficiency may still be an issue, as it comes down to solving a system of $n$ equations in $n$ variables. 
The second method, which is a specialisation of the before-mentioned iterative computation of the inverse, seems to be rather efficient.
This algorithm computes a parametrized preimage curve (if it exists) to a given paramtetrized curve $g$, i.e. if $g(t):k\lp k^n$ is given, the algorithm computes $f(t)$ such that 
$F(f(t))=g(t)$. The advantage of the Groebner bases algorithm is that 
the latter  gives a criterion to decide if there is no preimage curve. 

We point out how this might  affect cryptographic systems like the TTM method (positively or negatively).

\section{Iterative computation of the inverse}
\label{three}

\subsection*{Examples and background}

Suppose that $F=I-H\in \MA_n(R)$ where $H\in (\a A)^{{\times n}}\subset \MA_{n}(R)$ where $\a$ is an ideal of $A$. The following is well-known:

\begin{proposition}\label{Inv-2}
Let $\a$ be an ideal of $A$ such that $\cap \a^n=0$.  
Let $F=I-H$ where $H\in \a A^{\times n}$.   Then $F$
has an inverse in the $\a$-adic completion of $A$.
\end{proposition}

The above proposition is often applied for the case that $A=k[X_1,\ldots, X_n]$ and $\a=(X_1,\ldots,X_n)A$, and $=I-H$ where $H\in \a^2 A^{\times n}$.  
For completeness sake, we indicate how this $\a$-adic inverse is computed and how it can yield the inverse if it exists. 
If $I+K$ is the inverse, then it is clear that $K\in \a^2 A^{\times n}$. Now inductively, if the coefficients of $K$ up to and includig degree $d$ are known, giving a map $K_d$ which matches $K$ up to and including degree $d$, then $(I-H)(I+K) \mod \a^{d+1}= I$. 
Putting the coefficients of $K$ of degree $d+1$ as variables, and computing $(I+H)(I+K)$ modulo $\a^{d+2}$, 
yields a system of linear equations in $R$ that is always solvable. 
In particular, if $I+H$ has a polynomial inverse, then at some point in this process one will have the inverse. 

The interesting thing is that one can also apply such a technique for other ideals  $\a\subset A$; for example, 
if $H=(2X_2+X_2^2, 0)$, $R=\Z$, then one can compute the inverse in the $(2,X_1,X_2)$-adic completion, which in this case again describes the actual inverse of $I-H$. However, things get tricky - what are the requirements on $H$ to make this work? If an inverse exists, can one just approximate it or actually give an inverse? For example, if $R=\C[[t]]$, $F=X_1-tX_1$ and one starts to compute coefficients of an inverse in the $t$-adic completion, then there will be no point in the computation where the coefficient will be known (the coefficient is $(1-t)^{-1}$ while after $m$ steps one has $1+t+t^2+\ldots +t^m$).

In this section, we describe a slightly different method to iteratively compute the power series inverse, that is at first not necessarily more efficient, 
but has some conceptual value that we will see later on. Next to that, we will give the abstract setting in which this (and the power-series method) works for other cases than the ideal $(X_1,\ldots,X_n)$. The very rough, unpolished, basic algorithm (which never stops in this form)  is the following:\\
\ \\
{\bf Algorithm 1:} Suppose $I-H\in \MA_n(R)$ is given.\\
(1) Let $d=0$ and choose $K_0\in \MA_{n}(R)$ arbitrary (standard choice is $K_0=0$).\\
(2) Define $K_{d+1}:=H(I+K_d)$.\\
(3) Increase $d$, goto (2).\\

We will discuss later under what condition this algorithm makes sense and works - the idea is that $K_d$ converges to $K$ such that $I+K$ is the inverse.
A working example for later reference:


\begin{example}\label{One}
Define $A_i:=(X_1,\ldots,X_n)^{i+1}A$ and assume $H\in A_1$. Let $I+K$ be the formal power series inverse of $I-H$. Choose $K_0=0$ and define $K_i$ as above. Then $K\bmod{A_i} = K_i \bmod{A_i}$.  In particular, if $I-H$ is indeed invertible, then $K_i\bmod{A_i}$ ``equals'' $K$, where this ``equals'' means that taking the element in $\MA_n(R)$ which has the same coefficients as $K_i$ up to degree $i$, and zeros from  degree $i+1$ on, then this is equal to $K$. 
\end{example}

\subsection*{When does iteration leads to an inverse in finitely many steps?}

This section is more abstract than section \ref{vier} and beyond - the reader interested in the more applicable aspects of this paper can forward to section \ref{vier}. Also, it may be helpful to keep the (most important) example \ref{One} where $\a=(X_1,\ldots, X_n)A$, $A_0=\a A, A_1=\a^2A,A_2=\a^3 A,\ldots$ in mind when reading 
the below definitions:\\. 

Suppose $A\supseteq A_0\supseteq A_1\supseteq \ldots $ is a descending chain of ideals such that $\bigcap A_i =(0)$. 
We denote the projections $\pi_d: A\lp A/A_d$ as well as $\pi^{d+e}_d: A/A_{d+e}\lp A/A_d$. 
We assume that for each $d$ we have a section $s_d:A/A_d\lp A$, i.e. $\pi_d s_d (a)=a$ for all $a\in A/A_d$.\\

\begin{definition}   We call $A\supseteq A_0\supseteq \ldots $ a {\em composition-filtration} if for any 
$H\in (A_1)^n, G,\tilde{G}\in (A_0)^n$ we have: $\pi_d(G)=\pi_d(\tilde{G}) \lp \pi_{d+1}(H(G))=\pi_{d+1}(H(\tilde{G}))$. \\
We say that the $s_d$ form a {\em converging system of sections}\footnote{The authors did not find any already existing term in the literature.} if for all $a\in A$ there exists $D\in \N$ such that if $d\geq D$, then  
$s_d\pi_d(a)=a$.
\end{definition}

Let us explain how the above definition appears in example \ref{One}. Here, $A_i:=(X_1,\ldots,X_n)^{i+1}A$. This indeed is a composition-filtration as can be easily verified (substituting something having no terms below degree $d>0$ into something having no terms below degree $2$  yields something having only terms of degree $2d$ or higher). The sections $s_d$ here are the obvious canonical bijective map sending $A/A_d$ to the elements in $A$ of degree $\leq d$. 
Indeed, given $a\in A$, then one can take $D:=\deg{a}$, showing that this set of sections is a converging system of sections.


We define the following abbreviation of assumptions:\\
\ \\
{\em {\bf (P)} stands for the following list of assumptions: 
$A_0\supseteq A_1\supseteq \ldots$ is a composition-filtration, and 
we have a returning system of sections $s_i:A_i\lp A$. Let $F=I-H$ and $F^{-1}=I+K$.
Assume $H\in (A_1)^n$, $I\in (A_0)^n$ the identity map. }

\subsection*{The iterative inverse algorithm}


\begin{definition} Define $\varphi: \MA_n(R)\lp \MA_n(R)$ by $\varphi(K):=H(I+K)$. 
\end{definition}

\begin{lemma}\label{L1}
Assume (P). 
Let $\tilde{K}\in (A_0)^n\subseteq \MA_n(R)$. 
If $\pi_{d} \tilde{K}=\pi_d K$, then $\pi_{d+1}\varphi{\tilde{K}}=\pi_{d+1}K$. 
\end{lemma}

\begin{proof}
Because we have a composition-filtration, $\pi_{d}(I+ \tilde{K})=\pi_d (I+K)$
implies $\pi_{d+1}(H(I+\tilde{K}))=\pi_{d+1}(H(I+K))$. We claim that the latter equals $\pi_{d+1}(K)$: since
$I=(I-H)(I+K)=I+K-H(I+K)$ we have $H(I+K)=K$.
\end{proof}

\begin{corollary}
Assuming (P), the chain $0=K_0, K_{d+1}:=s_{d+1}\pi_{d+1}\varphi K_d$ stabilises.
\end{corollary}

\begin{proof}
First we give a proof by induction on $d$ to show that $\pi_d K_d=\pi_d K$.(*) This statement is obviously true for $d=0$. Assuming $\pi_d K_d= \pi_d K$, we get by lemma \ref{L1}
that $\pi_{d+1}K=\pi_{d+1}\varphi K_d$. Since $\pi_{d+1}s_{d+1}\pi_{d+1}=\pi_{d+1}$ for every $d$,
 $\pi_{d+1}K=\pi_{d+1} \varphi K_d= \pi_{d+1}s_{d+1}\pi_{d+1}\varphi K_d=\pi_{d+1} K_{d+1}$ (end induction).

Now $s_d\pi_dK_{d}=s_d\pi_d s_d \pi_d \varphi K_{d-1}$ and since $\pi_ds_d$ is the identity this equals $s_d\pi_d \varphi K_{d-1}=K_d$, thus $s_d\pi_d K_d=K_d$ (**).
Since we have a returning system of sections, we have some $D\in \N$ such that if
$d\geq D$ then $s_d\pi_d K=K$ (***). Thus, 
\[ K_d\overset{(**)}{=}s_d\pi_d K_d\overset{(*)}{=}s_d \pi_d K\overset{(***)}{=}K \]
 whenever $d\geq D$ (only to ensure (***)).
\end{proof}

The above corollary thus gives an algorithm, which we now denote separately:\\
\ \\
{\bf Algorithm 2:} Assume (P). Input $H\in A_1 $.\\
(1) Let $d=0$ and $K_0=0\in A^n$.\\
(2) Define $K_{d+1}:=s_{d+1}\pi_{d+1} H(I+K_d)$.\\
(3) If $K_d=K_{d+1}$, and $K_{d-1}\not = K_d$, then check if $H(I+K_d)=K_d$. If YES then STOP; output $I+K_d$.\\
(4) Increase $d$, goto (2).\\

\subsection*{More examples}

\begin{example}
$A:=\Z[X_1,\ldots,X_n]$,  and  $A_i:=2^iA$. 
Let $F=(x+2y+4x^2,y+2x^2)$ and thus $H=(2y+4x^2,2x^2)$ in $(A_1)^2$.
One can check that this is indeed a composition-filtration.
The sections $s_d:A/A_d \lp A$ must be chosen a bit carefully: 
we know that the inverse of $F$ will have coefficients that are ``not far from zero'', i.e.\@ there is a bound $D$ for which the coefficients must be in the interval $[-D,D]$. Therefore, we take the section map $s$ that sends elements of $\Z/(2^d\Z)$ into the interval $[-2^{d-1}, 2^{d-1}-1]$, which is indeed a returning section. If one chooses the interval $[0,2^k-1]$ as is custom, it is not a returning section.

Now the iteration process yields $K_0:=(0,0), K_1=K_0, K_2=(-2y,2x^2), K_3=K_2.$
The algorithm in step 3 now checks if $I+K_3$ is the inverse of $I-H$, but it is not, so we continue. $K_4:=(-2y, -2x^2-8xy-8y^2), K_5:=(-2y,-2x^2+8xy-8y^2), K_6=K_5$ and $I+K_5$ turns out to be the inverse. 
\end{example}

\begin{example}
Let $F\in \GA_n(\FF)$ be such that the linear part of $F$ is $I$. For example, let $F=(X+Y^2+2X^2Y+X^4, Y+X^2)$. Let $H:=F-I$. 
We define $A_d:=(X,Y)^{d+1}\FF[X,Y] \subset \FF[X,Y]$. 
Now $K_0:=(0,0)=K_1=K_2, K_3=(Y^2,X^2), K_4=(Y^2,X^2-2XY^2), K_5=(Y^2,X^2-2XY^2+Y^4)$ and since $K_6=K_5$ it is time to check if this might be the inverse (otherwise one has to continue). Indeed, $(I-K_5)F=I$. 
\end{example}

In the case that $A=R^{[n]}$ where $R$ is a reduced $\FF$-algebra, and \\
$A_d=(X_1,\ldots,X_n)^{d+1}A$, the algorithm is effective in deciding if a map is invertible. 
This is due to the theorem that  $\deg(F^{-1})\leq \deg(F)^{n-1}$ if $R$ is a reduced ring (corollary 2.3.4 in \cite{Essenboek}).

\section{Injective morphisms}
\label{vier}

\subsection*{Iterative preimage algorithm}

\label{four}

In this section, we will assume that $F:R^n\lp R^n$ is a polynomial endomorphism of the form $F=I-H$ where $H$ has affine part zero. 
Suppose $g(t):=(g_1(t),\ldots,g_n(t))\in (R[t])^n$ is a nonzero curve satisfying $g(0)=0$, and $f(t):=F(g(t))$, which hence is a curve contained in the image of $F$. 
(Note that $f(0)=F(g(0))=F(0)=0$.)
Since $F$ is of the described form, its extension $F:R[[t]]^n\lp R[[t]]^n$ is an automorphism. Hence, there is at most one parametrized curve $\tilde g(t)$ satisfying $\tilde g(0)=0$ such that $F(\tilde g(t))=f(t)$.
(Note: being the image of such a parametrized curve may be something stronger as being a curve which is contained in the image of $F$!)
We will describe a method to compute the curve $g(t):=(g_1(t),\ldots,g_n(t))$ given $f(t)$ and $F$. 

\begin{remark} \label{remark} Given $F=I-H$ where the affine part of $H$ is zero, and $f(t)\in R[t]^n$ such that $f(0)=0$. Then there exists at most one 
$\tilde g(t)\in R[t]^n$ satisfying $\tilde g(0)=0$ such that $F(\tilde g)=f$. 
\end{remark}

\begin{proof}
Since $F$ is of the form $I-H$ where $H$ has affine part zero, it has a power series inverse $G$. If $f\in R[[t]]$ such that $f(0)=0$, then $g:=G(f)$ is a well-defined element of $R[[t]]$. Since in this case, $g=G(f)=G(F(\tilde g))=\tilde g$, $\tilde g$ is unique. In case $\tilde g\in R[t]^n$, there is one solution, if $\tilde g \in R[[t]]^n\backslash R[t]^n$ there is none.
\end{proof}

\noindent
{\bf Algorithm 3:} $F,f$ as above. \\
(1) Let $d=1$ and $K_1=0\in R^n$.\\
(2) Define $K_{d+1}:= H(f+K_d) \mod(t^{d+1})$\\
(3) If $K_d=K_{d+1}$, and $K_{d-1}\not = K_d$, then check if $H(f+K_d)=K_d$. If YES then STOP; output $f+K_d$.\\
(4) Increase $d$, goto (2).\\

\begin{proposition}
If $f\in R[t]^n$ satisfying $f(0)=0$ and there is some $g\in R[t]^n, g(0)=0$ such that $F(g)=f$, then the above process terminates, and the output equals $g$.  Furthermore,
$g$ is unique. 
\end{proposition}

\begin{proof}
Uniqueness follows from remark \ref{remark}. We will prove that $K_d\equiv g-f \bmod{t^{d}}$.
The case $d=1$ is trivial.
Assume $K_d\equiv g-f  \bmod{t^{d}}$. Then $K_{d+1} = H(f+K_d)$.
Now remark that since $H$ has affine part zero, then for any $p,q\in R[t]$ satisfying $p(0)=q(0)=0$, we have $p\equiv q  \bmod {t^d} \Rightarrow H(p) \equiv H(q) \bmod t^{d+1}$.
Note that $f+K_d\equiv g \bmod t^{d}$, hence $H(f+K_d)\equiv H(g) \bmod t^{d+1}$. Since $f=F(g)=(I-H)(g)=g-H(g)$, we have $H(g)=g-f$. 
Concluding, $K_{d+1}=H(f+K_d)\equiv g-f \bmod t^{d+1}$. The proposition now follows.
\end{proof}

In case $F$ is an automorphism, there is obviously no need to require that $f=F(g)$ for some $g$; one only needs to assume that $f(0)=0$, for then 
$g:=F^{-1}(f)$ satisfies $g(0)=0$. 

\begin{remark}
If $F$ is an automorphism, then a preimage of $c\in R^n$ can be computed by computing the preimage curve $g(t)$ of $ct:=(c_1t,\ldots,c_nt)$, and then 
$g(1)$ is the preimage of $c$ (since $F(g(t))=ct$). (One could take any curve $f$ through $c$ satisfying $f(0)=0$, though.) Our experiments have shown this setting to be quite efficient. 
\end{remark}

\subsection*{Groebner bases preimage algorithm}

In this section we give another method to compute preimages of points and curves under polynomial automorphisms. We stick to the case where $R=k$, a field. 

In \cite{Essenboek} theorem 3.2.1/ 3.2.3 (page 64) an algorithm is given to compute the inverse (and effectively decide if an endomorphism is an automorphism). We will quote the case we will need here:

\begin{theorem}[van den Essen] Let $F\in (k[X_1,\ldots,X_n])^n$ be a polynomial endomorphism. Let $I=(Y_1-F_1,\ldots,Y_n-F_n)$ be 
an ideal in $k[X_1,\ldots,X_n, Y_1,\ldots, Y_n]$. 
Let $B$ be the reduced groebner basis of $I$ with respect to an ordering where $Y^{\alpha} < X_i$ for each $\alpha\in \N^n, 1\leq i\leq n$. 
$F$ is invertible if and only if $B$ is of the form $(X_1-G_1(Y),\ldots,X_n-G_n(Y))$, and in that case $G:=(G_1,\ldots,G_n)$ is the inverse of $F$.
\end{theorem}

The following is straightforward:

\begin{corollary}
Let $F\in (k[X_1,\ldots,X_n])^n$ be a polynomial endomorphism. Let $I=(c_1-F_1,\ldots,c_n-F_n)$ where $c_i\in k$ be 
an ideal in $k[X_1,\ldots,X_n]$. 
Let $B$ be the reduced groebner basis of $I$. Then \\
(1) $B=(X_1-b_1,\ldots,X_n-b_n)$ if and only if $F(X_1,\ldots,X_n)=c$ has only one solution $X_i=b_i$.  \\
(2) If $B$ is not of the form in (1) then $F$ is not an automorphism.
\end{corollary}

We will give a modified version of the above theorem of van den Essen to find preimages of curves.

\begin{definition} Let $F=I-H \in k[X_1,\ldots,X_n]^n$ where $H$ has affine part zero, and  $f,g\in k[t]^n$ such that $f(0)=g(0)=0$.
Then we define the following ideals in  $\C[t][X_1,\ldots,X_n]$: $(F-f):=(F_1-f_1,\ldots,F_n-f_n)$ and $(X-g):=(X_1-g_1,\ldots,X_n-g_n)$. 
\end{definition}

The only reason we assume that $F$ is of the described form $I-H$ is because we can then use remark \ref{remark}.

\begin{theorem} \label{groebner} Let $F\in (k[X_1,\ldots,X_n])^n$ be a polynomial endomorphism, and let $f(t)$ be a curve. 
Let $B$ be the reduced groebner basis of $(F-f)$ with respect to an ordering where $t^m < X_i$ for each $m\in \N$,$ 1\leq i\leq n$. 
Now \\
(1) $B=(X_1-g_1,\ldots,X_n-g_n)$ if and only if $(F-f)=(X-g)$. In particular, if $F$ is an automorphism, then $B$ is of the said form.  \\
(2) If $F(g)=f$, then $B\subseteq (X-g)$.   Hence, a curve $g$ such that $F(g)=f$ can, if it exists, be found by finding an ideal $(X-g)\supseteq B$.
\end{theorem}

Note that in part (2), finding such a $g$ may have become much easier because of the simpler form of $B$ compared to $(F-f)$.

Theorem \ref{groebner} is based on the following lemma:

\begin{lemma}\label{groebhelp}
(1) $(F-f)\subseteq (X-g)$ $\Leftrightarrow$ $F(g)=f$. \\
(2) In case $F\in \GA_n(k)$, then we have
$(F-f) \subseteq (X-g) \desda (F-f) =(X-g) \desda F(g)=f$. 
\end{lemma}

\begin{proof}
$(F-f)\subseteq (X-g) \desda (F-f)\equiv 0 \bmod (X-g) \desda F_i(X)-f_i\equiv 0 \mod (X-g)$ for all $1\leq i\leq n$ $\desda F_i(g)-f_i=0$ for all $1\leq i\leq n$ $\desda F(g)=f$, proving (1).\\
If $F$ is invertible, then let $G$ be the inverse of $F$. By (1), $(F-f)\subseteq (X-g)\desda F(g)=f$, but also $G(f)=g$ hence $(G-g)\subseteq (X-f)$. Substituting $X:=F$ in the latter yields
$(X-g)\subseteq (F-f)$, proving (2).
\end{proof}

\begin{proof}[Proof of theorem \ref{groebner}]
(1) Suppose $(F-f)=(X-g)$. Then, since $(X_1-g_1,\ldots,X_n-g_n)=(X-g)$ is a reduced basis of $(F-f)$, this must be the result of the algorithm.
The other way around, if $B=(X_1-g_1,\ldots,X_n-g_n)$, then of course $(F-f)=(X_1-g_1,\ldots,X_n-g_n)=(X-g)$ since it's the same ideal, only a different basis.\\
(2) is just a reformulation of lemma \ref{groebhelp} part (1).
\end{proof}

\noindent
{\bf Maple files of algorithms:}
If you are interested in maple files using the iterative preimage algorithm, contact the authors.

\end{document}